\documentclass[12pt]{amsart}

\usepackage{amssymb,amsthm,amsmath}

\newcommand{\dd}{\mathrm{d}}
\newcommand{\E}{\mathbb{E}}
\newcommand{\1}{\textbf{1}}
\newcommand{\R}{\mathbb{R}}
\newcommand{\p}[1]{\mathbb{P}\left( #1 \right)}
\newcommand{\scal}[2]{\left\langle #1, #2 \right\rangle}

\newtheorem{lemma}{Lemma}
\newtheorem{theorem}{Theorem}

\theoremstyle{remark}
\newtheorem*{remark*}{Remark}

\theoremstyle{definition}

\newtheorem*{conjecture*}{Conjecture}

\usepackage[colorinlistoftodos]{todonotes}

\usepackage[paper=a4paper, left=1.2in, right=1.2in, top=1in, bottom=1in]{geometry}
\begin{document}

\title{A multidimensional analogue of the Rademacher-Gaussian tail comparison}

\author{Piotr Nayar and Tomasz Tkocz}
\thanks{PN supported in part by NCN grant DEC-2012/05/B/ST1/00412}

\begin{abstract}
We prove a dimension-free tail comparison between the Euclidean norms of sums of independent random vectors uniformly distributed in centred Euclidean spheres and  properly rescaled standard Gaussian random vectors.
\end{abstract}

\maketitle

{\footnotesize

}

\subsection*{Introduction}\label{sec:intro}

Tail comparison bounds, such as Hoeffding's inequality, have always played a crucial role in probability theory. When specified to concrete examples, very precise estimates for tail probabilities are usually known. For instance, if $\epsilon_1, \epsilon_2, \ldots$ are independent random variables each taking values $\pm 1$ with probability $\frac{1}{2}$ and $g_1, g_2, \ldots$ are independent standard Gaussian random variables, then for every $m \geq 1$, real numbers $a_1, \ldots, a_m$ and positive $t$,
\begin{equation}\label{eq:pin}\tag{P}
\p{|a_1\epsilon_1+\ldots+a_m\epsilon_m|>t} \leq c\cdot\p{|a_1g_1+\ldots+a_mg_m|>t}
\end{equation}
for some absolute constant $c$. This inequality was first proved by Pinelis in \cite{Pin} with $c \approx 4.46$. Talagrand in \cite{T} treated the case of independent (but not necessarily identically distributed) bounded random variables by means of the Laplace transform establishing similar Gaussian tail bounds. Bobkov, G\"otze and Houdr\'e obtained a bigger constant $c \approx 12.01$ in \eqref{eq:pin}, but their inductive argument was much simpler (see \cite{Bob}). Only very recently the best constant (equal approximately $3.18$) has been found (see \cite{Ben}). 

Oleszkiewicz conjectured the following multidimensional generalisation of Pinelis' Rademacher-Gaussian tail comparison \eqref{eq:pin}: fix $d \geq 1$, let $\xi_1, \xi_2, \ldots$ be independent random vectors uniformly distributed in the Euclidean unit sphere $S^{d-1}\subset \R^d$ and let $G_1, G_2, \ldots$ be independent standard Gaussian random vectors in $\R^d$ with mean zero and identity covariance matrix; there exists a universal constant $C$ such that for every $m \geq 1$, real numbers $a_1, \ldots, a_m$ and $t > 0$ we have
\begin{equation}\label{eq:KO}\tag{KO}
\p{\left\|\sum_{i=1}^ma_i\xi_i\right\|>t} \leq C\cdot\p{\left\|\sum_{i=1}^ma_i\frac{G_i}{\sqrt{d}}\right\|>t}.
\end{equation}
Here and throughout, $\|\cdot\|$ denotes the standard Euclidean norm in $\R^d$. 

Note that the normalisation is chosen so that the vectors $\xi_1$ and $G_1/\sqrt{d}$ have the same covariance matrix. Plainly, when $d=1$, \eqref{eq:KO} reduces to \eqref{eq:pin}. For general $d$, it is possible to deduce \eqref{eq:KO} with $C = O(\sqrt{d})$ from Theorem 2 in \cite{Led}. 

The goal of this note is to positively resolve Oleszkiewicz's conjecture. We shall show the following two theorems which are our main results. The latter will easily follow from the former.

\begin{theorem}\label{thm:spheres}
For every $d \geq 2$, inequality \eqref{eq:KO} holds with $C = C_0 = 397$.
\end{theorem}

\begin{theorem}\label{thm:general}
Let $X_1, X_2, \ldots$ be independent rotationally invariant random vectors having values in the unit Euclidean ball in $\R^d$. Let $G_1, G_2, \ldots$ be independent standard Gaussian random vectors in $\R^d$ with mean zero and identity covariance matrix. Then for every $m \geq 1$, real numbers $a_1, \ldots, a_m$ and $t > 0$ we have
\begin{equation}\label{eq:cor}
\p{\left\|\sum_{i=1}^ma_i X_i\right\|>t} \leq C_0\cdot\p{\left\|\sum_{i=1}^ma_i\frac{G_i}{\sqrt{d}}\right\|>t}, 
\end{equation}
where $\|\cdot\|$ stands for the standard Euclidean norm in $\R^d$ and $C_0=397$. 
\end{theorem}
\begin{remark*}
This will no longer hold if we only assume the boundedness of the $X_i$. For example, consider independent $X_i$ taking only two values $(\pm 1, 0, \ldots, 0)$ each with probability $\frac{1}{2}$. Then for, say $a_1 = \ldots = a_m = 1/\sqrt{m}$, $t = 2$, the right-hand side of \eqref{eq:cor} goes to zero when $d$ goes to infinity, whereas the left-hand side does not depend on $d$.
\end{remark*}

\subsection*{Acknowledgements}

The authors would like to thank Krzysztof Oleszkiewicz for introducing them into the subject. They are really grateful to Rafa\l \ \L ata\l a for a discussion concerning Theorem \ref{thm:general}.

\subsection*{Proofs}\label{sec:proofs}

Our proof of Theorem \ref{thm:spheres} is inductive, inspired by the inductive approach to the one dimensional case from \cite{Bob}. In the inductive step, using the spherical symmetry of our problem, we arrive at an inequality comparing the Gaussian volume between centred and shifted balls (Lemma \ref{lm:balls} below). This inequality can be viewed as a multidimensional generalisation of the two point-inequality derived in the inductive step in \cite{Bob}. Its proof leads us to somewhat subtle estimates for the Laplace transform of the first coordinate of $\xi_1$ (Lemma \ref{lm:Jd} below).

We shall need four lemmas. We start with a standard one which will be used to provide numerical values of our constants. We include its proof for completeness (see, e.g. Lemma 2 in \cite{LO}).

\begin{lemma}\label{lm:sqrtd}
Let $d \geq 2$ and $G$ be a standard Gaussian random vector in $\R^d$. Then
\[ 
\p{\|G\| > \sqrt{d}} \geq 1/33 \]
and
\[  
\p{\|G\| > \sqrt{d+2}} \geq 1/397.\]
\end{lemma}
\begin{proof}\renewcommand{\qedsymbol}{}
Let $G= (g_1,\ldots,g_d)$, where $g_1, \ldots, g_d$ are independent standard Gaussian random variables. Let $S = \sum_{i=1}^{d}(g_i^2-1)$. We want to estimate $\p{\|G\|>\sqrt{d}} = \p{S>0}$. Since $0 = \E S = \E S\1_{\{S>0\}} + \E S\1_{\{S\leq 0\}}$, using the Cauchy-Schwarz inequality we get
\[ 
\E|S| = \E S\1_{\{S>0\}} - \E S\1_{\{S \leq 0\}} = 2\E S\1_{\{S>0\}} \leq 2 \sqrt{\E S^2}\sqrt{\p{S>0}}.\]
Moreover, by H\"older's inequality
\[ 
\E S^2 = \E |S|^{4/3}|S|^{2/3} \leq \left(\E|S|^4\right)^{1/3}\left(\E|S|\right)^{2/3}. \]
Combining these bounds and using that $\E(g_i^2-1)^2 = 2$, $\E(g_i^2-1)^4 = 60$ we obtain
\[ 
\p{S>0} \geq \frac14 \frac{(\E|S|)^2}{\E S^2} \geq \frac{1}{4}\frac{(\E S^2)^2}{\E S^4} = \frac{1}{4}\frac{(2d)^2}{54d + 6d^2} = \frac{1}{\frac{54}{d}+6} \geq \frac{1}{33}. \]
For the second part, observe that
\begin{align*}
\p{\|G\|>\sqrt{d+2}} &\geq \p{g_d^2\geq 3, g_1^2+\ldots+g_{d-1}^2\geq d-1} \\
&=  \p{g_d^2\geq 3}\p{g_1^2+\ldots+g_{d-1}^2\geq d-1} \\
&\geq 2\p{g_d>\sqrt{3}}\cdot\frac{1}{33} > \frac{1}{397}.\ \Box
\end{align*}
\end{proof}

The next lemma gives tight estimates for the Laplace transform of the first coordinate of a random vector uniformly distributed in the unit sphere. We hope these estimates are of independent interest, in addition to playing a major role in our proof.

\begin{lemma}\label{lm:Jd}
For $d \geq -1$ and $b \geq 0$ let us denote 
\[
J_d = J_d(b) =\int_{-1}^{1} (1-x^2)^{d/2} e^{bx} \dd x. \]
Then for every $d \geq 2$ we have
\begin{itemize}
\item[(a)] $b^2J_{d+1} = -d(d+1)J_{d-1}+(d+1)(d-1)J_{d-3}$, 
\item[(b)] $\frac{J_{d-3}}{J_{d-1}} \geq \frac{d}{d-1}\left(\frac{1}{2} + \sqrt{\frac{1}{4}+\frac{b^2}{d(d+2)}}\right)$,
\item[(c)] $\frac{J_{d-1}}{J_{d+1}} \leq \frac{d+2}{d+1}\left(\frac{1}{2} + \sqrt{\frac{1}{4}+\frac{b^2}{d(d+2)}}\right)$,
\item[(d)] $J_{d+1}J_{d-3} \geq J_{d-1}^2 \frac{d(d+1)}{(d-1)(d+2)}$.
\end{itemize}
\end{lemma}
\begin{proof}
\emph{(a)} Integrating by parts twice we get $
b^2J_{d+1} = \int_{-1}^{1} \left(\sqrt{1-x^2}^{d+1}\right)'' e^{bx} \dd x$,
which, after computing the second derivative in the above expression, easily leads to the desired relation.

For the proof of \emph{(b)} and \emph{(c)} let us first observe that due to \emph{(a)} these two assertions hold true for $b=0$. We then show that for $b>0$, \emph{(b)} and \emph{(c)} are equivalent. Indeed, part \emph{(a)} yields 
\begin{equation}\label{j/j}
\frac{J_{d-3}}{J_{d-1}} = \frac{d}{d-1}+ \frac{b^2}{(d+1)(d-1)}\frac{J_{d+1}}{J_{d-1}}, \qquad d \geq 2.
\end{equation}
Thus, \emph{(b)} is equivalent to 
\[
\frac{b^2}{(d+1)(d-1)}\frac{J_{d+1}}{J_{d-1}} \geq \frac{d}{d-1}\left(-\frac{1}{2} + \sqrt{\frac{1}{4}+\frac{b^2}{d(d+2)}}\right) 
= \frac{d}{d-1}\frac{\frac{b^2}{d(d+2)}}{\frac{1}{2} + \sqrt{\frac{1}{4}+\frac{b^2}{d(d+2)}}}.
\]
After cancelling common factors on both sides this becomes \emph{(c)}.

Let us fix $b > 0$. We shall show \emph{(b)} by backwards induction on $d$. We can use \eqref{j/j} for $d+2$, that is the equality 
\[ 
\frac{J_{d-1}}{J_{d+1}} = \frac{d+2}{d+1}+ \frac{b^2}{(d+3)(d+1)}\frac{J_{d+3}}{J_{d+1}}, \]
to rewrite \emph{(c)} in the form
\[
\frac{b^2}{(d+3)(d+1)}\frac{J_{d+3}}{J_{d+1}} \leq \frac{d+2}{d+1}\left(-\frac{1}{2} + \sqrt{\frac{1}{4}+\frac{b^2}{d(d+2)}}\right) 
=\frac{d+2}{d+1}\frac{\frac{b^2}{d(d+2)}}{\frac{1}{2} + \sqrt{\frac{1}{4}+\frac{b^2}{d(d+2)}}},
\]
which becomes
\begin{equation}\label{eq:goal2}
\frac{J_{d+1}}{J_{d+3}} \geq \frac{d}{d+3}\left(\frac{1}{2} + \sqrt{\frac{1}{4}+\frac{b^2}{d(d+2)}}\right).
\end{equation}
First notice that \emph{(b)} and equivalently \eqref{eq:goal2} hold for all $d \geq d_0(b)$ for some large enough $d_0(b)$ which depends only on $b$. To see this observe that the left-hand side of \eqref{eq:goal2} is strictly greater than $1$, whereas the right-hand side for large $d$ is of order $1-3/d+o(1/d)$. Now suppose \emph{(b)} holds for $d+4$, that is
\[ 
\frac{J_{d+1}}{J_{d+3}} \geq \frac{d+4}{d+3}\left(\frac{1}{2} + \sqrt{\frac{1}{4}+\frac{b^2}{(d+4)(d+6)}}\right) \]
and we want to show \emph{(b)} (induction step). By the above and the fact that \eqref{eq:goal2} and \emph{(b)} are equivalent, it is enough to show that
\[ 
\frac{d+4}{d+3}\left(\frac{1}{2} + \sqrt{\frac{1}{4}+\frac{b^2}{(d+4)(d+6)}}\right) \geq \frac{d}{d+3}\left(\frac{1}{2} + \sqrt{\frac{1}{4}+\frac{b^2}{d(d+2)}}\right). \]
This follows from $\frac{d+4}{d+6} \geq \frac{d}{d+2}$ and the estimate
\begin{align*}
\frac{d+4}{2} + \sqrt{\frac{(d+4)^2}{4}+b^2\frac{d+4}{d+6}} \geq \frac{d}{2} + \sqrt{\frac{d^2}{4}+b^2\frac{d}{d+2}}.
\end{align*}

Clearly \emph{(d)} immediately  follows from \emph{(b)} and \emph{(c)}.
\end{proof}

\begin{remark*}
Part \emph{(d)} improves on H\"older's inequality which gives $J_{d-1}^2 \leq J_{d+1}J_{d-3}$.
\end{remark*}

\begin{remark*}
Let us define for $d \geq -1$ and $b \geq 0$ the normalised integrals $\bar{J}_d(b) = J_d(b)/J_d(0)$ so that they are the Laplace transforms of the probability densities: if $d \geq 2$ and $\xi$ is a random vector in $\R^d$ uniformly distributed in the Euclidean unit sphere $S^{d-1}$, we check that (by rotational invariance)
\[ 
\bar{J}_{d-3}(b) = \E e^{\scal{v}{\xi}}, \]
for any vector $v \in \R^d$ of length $b$. Part \emph{(a)} for $b=0$ gives $J_{d-3}(0)/J_{d-1}(0)=d/(d-1)$. This allows to simplify \emph{(b),(c),(d)} rewritten in terms of $\bar{J}_d$ to get for $d\geq 2$
\begin{itemize}
\item[(a')] $\frac{b^2}{d(d+2)}\bar{J}_{d+1} = -\bar{J}_{d-1}+\bar{J}_{d-3}$,
\item[(b')] $\frac{\bar{J}_{d-3}}{\bar{J}_{d-1}} \geq \frac{1}{2} + \sqrt{\frac{1}{4}+\frac{b^2}{d(d+2)}}$,
\item[(c')] $\frac{\bar{J}_{d-1}}{\bar{J}_{d+1}} \leq \frac{1}{2} + \sqrt{\frac{1}{4}+\frac{b^2}{d(d+2)}}$,
\item[(d')] $\bar{J}_{d+1}\bar{J}_{d-3} \geq \bar{J}_{d-1}^2$.
\end{itemize}
\end{remark*}

The following lemma lies at the heart of our inductive argument. It compares the standard Gaussian measure of centred and shifted Euclidean balls.

\begin{lemma}\label{lm:balls}
Let $d \geq 2$ and $G$ be a standard Gaussian random vector in $\R^d$. For every $a \geq 0$, $R \geq \sqrt{d+2}$ and a vector $x \in \R^d$ of length $a\sqrt{d}$ we have
\[ 
\p{\|G\|\leq R} \leq \p{\|G-x\|\leq R\sqrt{1+a^2}}. \]
\end{lemma}
\begin{proof}
Since for $a=0$ we have equality, it is enough to show that the right-hand side,
\[h(a,R) = \p{\|G-a \sqrt{d} e_1\|\leq R\sqrt{1+a^2}}\]
is nondecreasing with respect to $a$ (by rotational invariance, for concreteness we can choose $x = ae_1$, where $e_1=(1,0,\ldots,0)$). Using Fubini's theorem we can write
\begin{align*}
h(a,R) = \frac{|S^{d-2}|}{\sqrt{2\pi}^{d-1}}\int_{0}^{R\sqrt{1+a^2}}\bigg\{r^{d-2}e^{-r^2/2}\int_{a\sqrt{d}-\sqrt{R^2(1+a^2)-r^2}}^{a\sqrt{d}+\sqrt{R^2(1+a^2)-r^2}}\phi(t)\dd t\bigg\} \dd r,
\end{align*}
where $\phi(t) = \frac{1}{\sqrt{2\pi}}e^{-t^2/2}$.
The derivative with respect to $a$ equals
\begin{align*}
\frac{\partial}{\partial a}h(a,R) &= \frac{|S^{d-2}|}{\sqrt{2\pi}^{d-1}}\int_{0}^{R\sqrt{1+a^2}}\bigg\{r^{d-2}e^{-r^2/2}\\
&\cdot\bigg[\phi\big(a\sqrt{d}-\sqrt{R^2(1+a^2)-r^2}\big)\left(\sqrt{d}+\frac{aR^2}{\sqrt{R^2(1+a^2)-r^2}}\right)\\
&-\phi\big(a\sqrt{d}+\sqrt{R^2(1+a^2)-r^2}\big)\left(\sqrt{d}
-\frac{aR^2}{\sqrt{R^2(1+a^2)-r^2}}\right)\bigg]\bigg\} \dd r.
\end{align*}
After changing the variables $r = R\sqrt{1+a^2}\sqrt{1-x^2}$ we see that this is nonnegative if and only if
\begin{align*}
\int_{0}^{1} \sqrt{1-x^2}^{d-3}\Bigg[&e^{-xa\sqrt{d}R\sqrt{1+a^2}}\left(x\sqrt{d}+\frac{aR}{\sqrt{1+a^2}}\right)\\
&-e^{xa\sqrt{d}R\sqrt{1+a^2}}\left(x\sqrt{d}-\frac{aR}{\sqrt{1+a^2}}\right)\Bigg] \dd x \ \geq  \ 0.
\end{align*}
This condition can be further simplified by integration by parts using $\left(\sqrt{1-x^2}^{d-1}\right)' = -(d-1)x\sqrt{1-x^2}^{d-3}$. We obtain an equivalent inequality
\[ 
\int_{0}^{1} \sqrt{1-x^2}^{d-3}\left(d-1-d(1+a^2)(1-x^2)\right)\cosh\left(aR\sqrt{d}\sqrt{1+a^2}x\right)\dd x \geq 0. \]
Let $b = a\sqrt{d}R\sqrt{1+a^2}$. Then
\[ 
1+a^2 = \frac{1}{2} + \sqrt{\frac{1}{4}+\frac{b^2}{R^2d}}. \]
Observe that
\[ 
  \int_{0}^{1} \sqrt{1-x^2}^d\cosh(bx) \dd x = \frac12 \int_{-1}^{1} \sqrt{1-x^2}^de^{bx} \dd x = \frac12 J_d(b). \]
Thus, the inequality we want to show becomes
\[ 
\frac{J_{d-3}(b)}{J_{d-1}(b)} \geq \frac{d}{d-1}\left(\frac{1}{2} + \sqrt{\frac{1}{4}+\frac{b^2}{R^2d}}\right). \]
For a fixed $b$, the right-hand side as a function of $R$ is clearly decreasing, so given our assumption $R \geq \sqrt{d+2}$ it is enough to consider $R = \sqrt{d+2}$, which follows from Lemma \ref{lm:Jd}\emph{(b)}.
\end{proof}

\begin{remark*}
The statement for $d=1$ remains true and was proved in \cite{Bob}, where it played a key role in the inductive proof of Pinelis' inequality \eqref{eq:pin}.
\end{remark*}

The last lemma will help us use the spherical symmetry of our problem.

\begin{lemma}\label{lm:expand}
Let $X$ be a rotationally invariant random vector in $\R^d$. Let $x \in \R^d$ and $t > 0$ be such that $t > \|x\|$. Then
\[ 
\p{\|X+x\|>t} = \p{\|X\| > -\theta\|x\| + \sqrt{t^2+\theta^2\|x\|^2-\|x\|^2}}, \]
where $\theta$ is the first coordinate of an independent of $X$ random vector uniformly distributed in the unit sphere $S^{d-1}$ in $\R^d$.
\end{lemma}
\begin{proof}
Let $\xi$ be an independent of $X$ random vector uniformly distributed in the unit sphere $S^{d-1}$ in $\R^d$. By rotational invariance $X$ has the same distribution as $R\xi$, where $R = \|X\|$. We have
\[ 
\p{\|X+x\|>t} = \p{R^2 + 2R\scal{\xi}{x} + \|x\|^2 > t^2} \]
and by the rotational invariance of $\xi$, $\scal{\xi}{x}$ has the same distribution as $\theta\|x\|$ with $\theta$ being the first coordinate of $\xi$. The inequality $R^2+2R\theta\|x\|+\|x\|^2>t^2$ is equivalent to $R > -\theta\|x\| + \sqrt{t^2+\theta^2\|x\|^2-\|x\|^2}$ or $R < -\theta\|x\| - \sqrt{t^2+\theta^2\|x\|^2-\|x\|^2}$, but the second case does not hold as the right-hand side is negative, for we assume that $t > \|x\|$.
\end{proof}

\begin{proof}[Proof of Theorem \ref{thm:spheres}]
We fix $d \geq 2$ and proceed by induction on $m$. For $m=1$ we have to check that for $0 < t < |a_1|$ we have
\[ 
1 \leq C_0\cdot\p{\left\|a_1\frac{G_1}{\sqrt{d}}\right\|>t}. \]
This follows because $\p{\|G_1\|>\sqrt{d}} \geq 1/33$ by Lemma \ref{lm:sqrtd}.

Suppose the assertion is true for $m \geq 1$. We shall show it for $m+1$. We can assume that the $a_i$ are nonzero. By homogeneity we can also assume that $\sum_{i=2}^{m+1} a_i^2 = d$. If $t \leq \sqrt{d+2}\sqrt{\frac{a_1^2+d}{d}}$ we trivially bound the right-hand side as follows:
\[ 
\p{\left\|\sum_{i=1}^{m+1}a_i\frac{G_i}{\sqrt{d}}\right\|>t} = \p{\left\|\sqrt{\frac{a_1^2+d}{d}}G_1\right\|>t} \geq \p{\|G_1\|>\sqrt{d+2}}, 
 \]
and by Lemma \ref{lm:sqrtd} we get $C_0\cdot\p{\left\|\sum_{i=1}^{m+1}a_i\frac{G_i}{\sqrt{d}}\right\|>t}\geq 1$.

Now suppose $t > \sqrt{d+2}\sqrt{\frac{a_1^2+d}{d}}$. Notice that in particular $t > |a_1|$. Consider $v = \sum_{i=2}^{m+1} a_i\xi_i$. By independence and rotational invariance,
\begin{align*}
\p{\left\|\sum_{i=1}^{m+1}a_i\xi_i\right\|>t} &= \mathbb{P}\left(\left\|a_1e_1+v\right\|>t\right).
\end{align*}
Lemma \ref{lm:expand} applied to $X = v$ and $x = a_1e_1$ yields
\begin{align*}
\p{\left\|\sum_{i=1}^{m+1}a_i\xi_i\right\|>t}
&= \E_\theta\mathbb{P}_v\left(\|v\| > -\theta|a_1| + \sqrt{t^2+\theta^2a_1^2-a_1^2}\right).
\end{align*}
As a consequence, by the independence of $\theta$ and $v$, and the inductive hypothesis,
\begin{align*}
\p{\left\|\sum_{i=1}^{m+1}a_i\xi_i\right\|>t} 
&\leq C_0\cdot\E_\theta\mathbb{P}_{(G_i)_{i=2}^{m+1}}\left(\left\|\sum_{i=2}^{m+1}a_i\frac{G_i}{\sqrt{d}}\right\| > -\theta|a_1| + \sqrt{t^2+\theta^2a_1^2-a_1^2}\right).
\end{align*}
The vector $\sum_{i=2}^{m+1}a_i\frac{G_i}{\sqrt{d}}$ has the same distribution as $\sqrt{\frac{\sum_{i=2}^{m+1}a_i^2}{\sqrt{d}}}G_1 = G_1$. Therefore, applying again Lemma \ref{lm:expand} yields
\[ 
\p{\left\|\sum_{i=1}^{m+1}a_i\xi_i\right\|>t} 
\leq C_0\cdot\p{\|G_1+a_1e_1\|>t}.\]
To finish the inductive step it suffices to show that
\[ 
\p{\|G_1+a_1e_1\|>t} \leq \p{\left\|\sum_{i=1}^{m+1}a_i\frac{G_i}{\sqrt{d}}\right\|>t} = \p{\left\|\sqrt{\frac{a_1^2+d}{d}}G_1\right\|>t}. \]
This follows from Lemma \ref{lm:balls} applied to $a = \frac{|a_1|}{\sqrt{d}}$ and $R = t\sqrt{\frac{d}{a_1^2+d}} > \sqrt{d+2}$, which completes the proof.
\end{proof}

\begin{proof}[Proof of Theorem \ref{thm:general}]\renewcommand{\qedsymbol}{}
Let $\xi_1, \xi_2, \ldots$ be independent random vectors uniformly distributed in the unit Euclidean sphere $S^{d-1}\subset \R^d$, independent of the sequence $X_1, X_2, \ldots$. Since $X_i$ is rotational invariant, it has the same distribution as $R_i \xi_i$, where $R_i=\|X_i\|$. Note that almost surely $0 \leq R_i \leq 1$. Applying \eqref{eq:KO} (conditionally on the $R_i$) we get
\begin{align*}
 \p{\left\|\sum_{i=1}^ma_i X_i\right\|>t} & = \E_{(R_i)_{i=1}^m}\mathbb{P}_{(\xi_i)_{i=1}^m}\left( \left\|\sum_{i=1}^ma_i R_i \xi_i \right\|>t \right) \\
 & \leq C_0\cdot \E_{(R_i)_{i=1}^m} \mathbb{P}_{(G_i)_{i=1}^m}\left(\left\|\sum_{i=1}^m a_i R_i\frac{G_i}{\sqrt{d}}\right\|>t\right).
\end{align*}
To finish the proof notice that for any fixed numbers $R_i \in [0,1]$ we have
\begin{align*}
\p{\left\|\sqrt{\frac{\sum_{i=1}^m a_i^2 R_i^2}{d}} G_1\right\|>t} \leq  \p{\left\|\sqrt{\frac{\sum_{i=1}^m a_i^2 }{d}} G_1\right\|>t}. \ \Box
\end{align*}
\end{proof}

{\footnotesize

\vspace{1em}

\begin{tabular}{lcl}

Piotr Nayar, \texttt{nayar@mimuw.edu.pl} & \hspace*{5em} & Tomasz Tkocz, \texttt{ttkocz@princeton.edu} \\
University of Pennsylvania & \ \ & Princeton University \\
Wharton Statistics Department & \ \ & Mathematics Department \\
3730 Walnut St & \ \ & Fine Hall \\
Philadelphia, PA 19104 & \ \ & Princeton, NJ 08544 \\
United States & \ \ & United States
\end{tabular}
\end{document}